\documentclass[12pt, reqno]{amsart}
\usepackage{amsmath, amsthm, amscd, amsfonts, amssymb, graphicx, color}
\usepackage[bookmarksnumbered, colorlinks, plainpages]{hyperref}

\newtheorem{theorem}{Theorem}[section]
\newtheorem{lemma}[theorem]{Lemma}
\newtheorem{proposition}[theorem]{Proposition}

\theoremstyle{definition}
\newtheorem{definition}[theorem]{Definition}

\theoremstyle{remark}
\newtheorem{remark}[theorem]{Remark}
\numberwithin{equation}{section}

\begin{document}
\setcounter{page}{1}

\title[2-Local derivations and automorphisms]{Description of 2-local derivations and automorphisms on finite dimensional Jordan algebras}

\author[Sh.\ Ayupov]{Shavkat Ayupov$^{1}$}

\address{$^{1}$
V.I.Romanovskiy Institute of Mathematics
Uzbekistan Academy of Sciences, Tashkent, Uzbekistan}
\email{\textcolor[rgb]{0.00,0.00,0.84}{sh$_-$ayupov@mail.ru}}

\author[F. Arzikulov]{Farhodjon Arzikulov$^{2,3}$}

\address{$^{2}$ V.I. Romanovskiy Institute of Mathematics, Namangan Regional Department, Uzbekistan Academy of Sciences}
\address{$^{3}$ Department of Mathematics, Andizhan State University, Andizhan, Uzbekistan.}
\email{\textcolor[rgb]{0.00,0.00,0.84}{arzikulovfn@rambler.ru}}

\author[N. Umrzaqov]{Nodirbek Umrzaqov$^4$}

\address{$^{4}$ Department of Mathematics, Andizhan State University, Andizhan, Uzbekistan.}
\email{\textcolor[rgb]{0.00,0.00,0.84}{umrzaqov2010@mail.ru}}

\author[O. Nuriddinov]{Olimjon Nuriddinov$^5$}

\address{$^{5}$ Department of Mathematics, Andizhan State University, Andizhan, Uzbekistan.}
\email{\textcolor[rgb]{0.00,0.00,0.84}{o.nuriddinov86@mail.ru}}

\subjclass[2010]{47A05, 17A36, 17C65, 47B47}

\begin{abstract}
In the present paper we introduce and investigate the notion of 2-local linear maps on  vector spaces.
A  sufficient condition is obtained for linearity of a 2-local linear map on  finite dimensional vector spaces.
Based on this result we prove that every 2-local derivation on a finite
dimensional semisimple Jordan algebra over an algebraically closed field of characteristic different from $2$
is a derivation. Also we show that every 2-local 1-automorphism (i.e., implemented by  single symmetries) of mentioned Jordan algebra is an automorphism.
\end{abstract}

\keywords{Vector space, Jordan algebra, 2-local linear map, 2-local Jordan derivation, 2-local 1-automorphism}

\maketitle

\footnotetext{}

\section{Introduction}
\label{sec1}

The present paper is devoted to the study of 2-local maps on vector spaces and Jordan algebras.
The Gleason-Kahane-\.{Z}elazko theorem, which is a fundamental contribution in the theory of Banach algebras,
asserts that every unital linear functional $F$ on a complex unital Banach algebra $A$, such that $F(a)$ belongs to the spectrum $\sigma(a)$
of $a$, for every $a\in A$, is multiplicative (cf. \cite{AMG}, \cite{JPK_WZ}). In modern terminology this is equivalent to the following condition: every unital linear local homomorphism from a unital complex Banach algebra $A$ into ${\Bbb C}$ is multiplicative. We recall that a linear map $T$ from a Banach algebra $A$ into a Banach algebra $B$ is said to be a local homomorphism if, for every $a$ in $A$, there exists a homomorphism $\Phi_a : A\to B$, depending on $a$,
such that $T(a)=\Phi_a(a)$.

In \cite{KS} S. Kowalski and Z. S{\l}odkowski give another characterization of multiplicative linear functionals in
Banach algebras. They prove that every 2-local homomorphism $T$ from a (not necessarily commutative nor unital) complex Banach
algebra $A$ into ${\Bbb C}$ is linear and multiplicative. Consequently, every (not necessarily linear) 2-local homomorphism $T$ from $A$ into
a commutative C$^*$-algebra is linear and multiplicative.

A similar notion was introduced and studied to give a characterization of derivations on operator algebras.
Namely, the notion of 2-local derivation was introduced by P. \v{S}emrl in his paper \cite{S} in 1997.
P. \v{S}emrl proved that a 2-local derivation on the algebra
$B(H)$ of all bounded linear operators on the infinite-dimensional
separable Hilbert space $H$ is a derivation. After a number of papers were devoted to
2-local maps on different types of rings, algebras, Banach algebras and Banach spaces.
The list of papers devoted to such 2-local maps can be found in the bibliography of \cite{AA3}.

In the present paper we study 2-local derivations on
Jordan algebras. Recall that a 2-local derivation is defined as follows:
given a Jordan algebra $\mathcal{A}$, a map $\Delta : \mathcal{A} \to \mathcal{A}$ (not linear in
general) is called a 2-local derivation if, for every $x$, $y\in
\mathcal{A}$, there exists a derivation $D_{x,y} : \mathcal{A}\to \mathcal{A}$ such that
$\Delta(x)=D_{x,y}(x)$ and $\Delta(y)=D_{x,y}(y)$.
2-local derivations on Jordan algebras are studied in \cite{AA3} by the first and the second authors
of the present paper. In \cite{AA3} it is proved that every 2-local inner derivation on the Jordan
ring $H_n(\Re)$ of symmetric $n\times n$ matrices over a commutative associative ring
$\Re$ is an inner derivation.

In this paper we develop an algebraic approach to the
investigation of derivations and 2-local derivations on Jordan algebras.

In section 2, we introduce the notion of
2-local linear map on a vector space and
give a sufficient condition for linearity of a 2-local linear map on a finite dimensional vector space.
A similar notion for Banach algebras was introduced and studied in
\cite{CaPe2015}.

In section 3, we show that
every 2-local inner derivation on a finite dimensional semisimple Jordan algebra over an algebraically closed
field of characteristic different from $2$ is a derivation. For this purpose first we prove linearity of
a 2-local inner derivation on this Jordan algebra, using Theorem \ref{0.2} of section 2,
and then we prove that this 2-local inner derivation is a derivation.
In \cite{CaPe2015} such approach was developed by J. Cabello and A. Peralta.

In section 4, we prove that every 2-local 1-automorphism of a finite
dimensional semisimple Jordan algebra over an algebraically closed
field of characteristic different from $2$ is an automorphism.
For this purpose we show the linearity of
a 2-local 1-automorphism of this Jordan algebra, using Theorem \ref{0.2} of section 2, and then we prove that this 2-local 1-automorphism is a global automorphism.

In section 5, we give another proof of assertions mentioned above in the case
of a finite dimensional formally real Jordan algebra. In particular, it is proved that
every 2-local derivation on a finite dimensional formally real exceptional simple Jordan algebra is a derivation
and every 2-local 1-automorphism on a finite dimensional formally real exceptional simple Jordan algebra is a global automorphism.

\section{2-Local linear maps on vector spaces}

\begin{definition} \label{0.0}
Let $V$ be a vector space over a field $\mathcal{F}$, $\phi : V \to V$ be a map
such that for each pair $v$, $w$ of elements in $V$ there exists an endomorphism $\phi_{v,w}$ of $V$
satisfying the following conditions
$$
\phi(v)=\phi_{v,w}(v), \phi(w)=\phi_{v,w}(w).
$$
Then $\phi$ is called 2-local linear map.
\end{definition}

\begin{definition} \label{0.1}
Let $V$ be a vector space of dimension $n$ over a field $\mathcal{F}$, $\phi$ be an endomorphism of
$V$, $\Lambda=(\lambda_{i,j})_{i,j=1,\dots,n}$ be an $n\times n$ matrix with non-zero entries of $\mathcal{F}$, $\mathcal{V}$ be a basis of $V$.
We say that $\phi$ is $\lambda$-symmetric with respect to $\mathcal{V}$, if the matrix $\alpha_{\mathcal{V}}(\phi)$ representing $\phi$ with respect
to $\mathcal{V}$ is of the form
$$
\alpha_{\mathcal{V}}(\phi)=(\lambda_{i,j}a_{i,j})_{i,j=1,\dots,n}
$$
where the matrix
$$
A=(a_{i,j})_{i,j=1,\dots,n}
$$
is symmetric.
\end{definition}

\begin{theorem}  \label{0.2}
Let $V$ be a vector space of dimension $n$ over a field $\mathcal{F}$, $\phi : V \to V$ be a 2-local
linear map, $\Lambda=(\lambda_{i,j})_{i,j=1,\dots,n}$ be an $n\times n$ matrix with non-zero entries of $\mathcal{F}$ and $\mathcal{V}$ be a basis of $V$.
Assume that for every $v$, $w$ in $V$ there is an endomorphism $\phi_{v,w}$ of $V$ which is $\Lambda$-symmetric with
respect to $\mathcal{V}$, and such that
$$
\phi(v)=\phi_{v,w}(v), \phi(w)=\phi_{v,w}(w).
$$
Then $\phi$ is linear. \end{theorem}

\begin{proof}
Let $\mathcal{V}=(v_1,...,v_n)$ and let $X=(x_{i,j})_{i,j=1,\dots,n}$, be defined by
$$
\phi(v_i)=\sum^n_{j=1}x_{j,i}v_j\,\,\,\, i=1,\dots,n
$$
For each $i$, $j=1,\dots,n$, let $a_{i,j}=\frac{x_{i,j}}{\lambda_{i,j}}$, hence
$$
\phi(v_i)=\sum^n_{j=1}\lambda_{j,i}a_{j,i}v_j\,\,\,\,    i=1,\dots,n
$$
We show that $A=(a_{i,j})_{i,j=1,\dots,n}$ is symmetric. Let $1\leq h<k\leq n$, and let $\phi_{v_h,v_k}$ be an
endomorphism of $V$ which is $\Lambda$-symmetric with respect to $\mathcal{V}$ such that $\phi(v_h)=\phi_{v_h,v_k}(v_h)$,
$\phi(v_k)=\phi_{v_h,v_k}(v_k)$. There is a symmetric matrix $B=(b_{i,j})_{i,j=1,\dots,n}$ such that
$$
\phi_{v_h,v_k}(v_i)=\sum^n_{j=1}\lambda_{j,i}b_{j,i}v_j\,\,\,\, i=1,\dots,n.
$$
It follows that
$$
\sum^n_{j=1}
\lambda_{j,h}a_{j,h}v_j=\phi(v_h)=\phi_{v_h,v_k}(v_h)=
\sum^n_{j=1}\lambda_{j,h}b_{j,h}v_j,
$$
$$
\sum^n_{j=1}\lambda_{j,k}a_{j,k}v_j=\phi(v_k)=\phi_{v_h,v_k}(v_k)=
\sum^n_{j=1}\lambda_{j,k}b_{j,k}v_j,
$$
hence
$$
\lambda_{j,h}a_{j,h}=\lambda_{j,h}b_{j,h}\,\,\,\, j=1,\dots,n,
$$
$$
\lambda_{j,k}a_{j,k}=\lambda_{j,k}b_{j,k}\,\,\,\, j=1,\dots,n.
$$

In particular, $\lambda_{k,h}a_{k,h}=\lambda_{k,h}b_{k,h}$ and
$\lambda_{h,k}a_{h,k}=\lambda_{h,k}b_{h,k}$, and since the $\lambda_{i,j}$'s are non-zero, it follows
that $a_{k,h}=b_{k,h}$ and $a_{h,k}=b_{h,k}$. But $B$ is symmetric, therefore $a_{k,h}=b_{k,h}=b_{h,k}=a_{h,k}$, and $A$ is
symmetric.

Now we show that $\phi$ is linear. Let $v=\sum^n_{i=1}x_iv_i$ in $V$, and let $\phi(v)=\sum^n_{i=1}y_iv_i$. To show
that $\phi$ is linear, we have to show that
$$
y_i=\sum^n_{j=1}\lambda_{i,j}a_{i,j}x_j\,\,\,\, i=1,\dots,n.
$$
Let $1\leq h\leq n$, and let $\phi_h$ be an endomorphism of $V$ which is $\Lambda$-symmetric with respect to $\mathcal{V}$
such that $\phi(v_h)=\phi_h(v_h)$, $\phi(v)=\phi_h(v)$. There is a symmetric matrix $B=(b_{i,j})_{i,j=1,\dots,n}$ such
that
$$
\phi_h(v_i)=\sum^n_{j=1}\lambda_{j,i}b_{j,i}v_j\,\,\,\,  i=1,\dots,n.
$$
It follows that
$$
\sum^n_{j=1}\lambda_{j,h}a_{j,h}v_j=\phi(v_h)=\phi_h(v_h)=
\sum^n_{j=1}\lambda_{j,h}b_{j,h}v_j,
$$
hence
$$
\lambda_{j,h}a_{j,h}=\lambda_{j,h}b_{j,h}\,\,\,\,  j=1,\dots,n.
$$
Since the $\lambda_{i,j}$'s are non-zero and $A$ and $B$ are symmetric, we get
$$
a_{h,j}=a_{j,h}=b_{j,h}=b_{h,j}\,\,\,\,  j=1,\dots,n.
$$
Since $\phi_h$ is linear, from
$$
\phi_h(v)=\phi(v)=\sum^n_{i=1}y_iv_i
$$
we get
$$
y_i=\sum^n_{j=1}\lambda_{i,j}b_{i,j}x_j,\,\,\,\,  i=1,\dots,n.
$$
In particular,
$$
y_h=\sum^n_{j=1}\lambda_{h,j}b_{h,j}x_j.
$$
But $a_{h,j}=b_{h,j}$, for $j=1,\dots,n$, and therefore
$$
y_h=\sum^n_{j=1}\lambda_{h,j}a_{h,j}x_j
$$
and this completes the proof.
\end{proof}

\section{2-Local derivations on finite dimensional semisimple Jordan algebras}

\begin{theorem} \label{2.1}
Let $\mathcal{J}$ be a finite dimensional Jordan algebra over a field $\mathcal{F}$ of characteristic different from $2$ and let $\Delta$ be a 2-local
derivation on $\mathcal{J}$, $\Lambda=(\lambda_{i,j})_{i,j=1,\dots,n}$ be an $n\times n$ matrix with non-zero entries of
$\mathcal{F}$, $\mathcal{V}$ be a basis of $\mathcal{J}$.
Assume that for every $x$, $y$ in $\mathcal{J}$ there is a derivation $D_{x,y}$ which is $\Lambda$-symmetric with
respect to $\mathcal{V}$, and such that
$$
\Delta(x)=D_{x,y}(x), \Delta(y)=D_{x,y}(y).
$$
Then $\Delta$ is a derivation.
\end{theorem}

\begin{proof}
By the definition $\Delta$ is a  $\Lambda$-symmetric 2-local linear map. Hence, by Theorem \ref{0.2},
$\Delta$ is a linear operator and $\Delta$
is a Jordan derivation, i.e., $\Delta(x^2)=\Delta(x)x+x\Delta(x)$, $x\in \mathcal{J}$.

Now, let $x$, $y$ be arbitrary elements in $\mathcal{J}$. Then
$$
4\Delta(xy)=\Delta((x+y)^2-(x-y)^2)=\Delta((x+y)^2)-\Delta((x-y)^2)
$$
$$
=2\Delta(x+y)(x+y)- 2\Delta(x-y)(x-y)
$$
$$
= 2(\Delta(x)+\Delta(y))(x+y)- 2 (\Delta(x)-\Delta(y))(x-y)= 4 (\Delta(x)y+x\Delta(y)).
$$
Hence $\Delta(xy)=\Delta(x)y+x\Delta(y)$. Therefore $\Delta$ is a derivation. The proof is complete.
\end{proof}

Let $V$ be a vector space over a field ${\Bbb F}$ which is equipped with a symmetric bilinear
form $f$. Thus $f(x,y)\in {\Bbb F}$, $f(x,y)=f(y,x)$ and $f(x_1+x_2,y)=f(x_1,y)+f(x_2,y)$,
$\alpha f(x,y)=f(\alpha x,y)$, $x_1$, $x_2$, $x$, $y\in V$, $\alpha\in {\Bbb F}$. We now
construct the vector space $\mathcal{J}Spin_n({\Bbb F})={\Bbb F}{\bf 1}\oplus V$ which is
a direct sum of $V$ and a one dimensional space ${\Bbb F}{\bf 1}$ with basis ${\bf 1}$. We define a product
in $\mathcal{J}Spin_n({\Bbb F})$ by
$$
(\alpha {\bf 1}+x)\cdot (\beta {\bf 1}+y)=(\alpha\beta+f(x,y)){\bf 1}+(\beta x+\alpha y)
$$
for $\alpha$, $\beta\in {\Bbb F}$, $x$, $y\in V$. Then $\mathcal{J}Spin_n({\Bbb F})$ is a Jordan algebra.
$\mathcal{J}Spin_n({\Bbb F})$ is called the Jordan algebra of the symmetric bilinear form $f$, also it is
called a spin factor. (cf. \cite[Chapter I, \S 4]{NJ}, \cite[Chapter 2]{KMcC})

Let $\{\xi_1, \xi_2, \dots, \xi_n\}$ be an orthogonal basis of $V$, i.e.,
$$
f(\xi_i,\xi_j)=0, \,\,\,\, i,j\in\{1,2,\dots,n\}.
$$
Put
$$
s_i=0\cdot {\bf 1}+\xi_i,\,\,\,\, i=1,2,\dots,n.
$$
Then
$$
\mathcal{V}=\{{\bf 1}, s_1, s_2, \dots, s_n\}
$$
is a basis of the vector space $\mathcal{J}Spin_n({\Bbb F})$ such that
$$
s_i\cdot s_i={\bf 1},\,\,\,\, s_i\cdot s_j=0, i\neq j,\,\,\,\, i,j=1,2,\dots, n,
$$
$$
s_k\cdot {\bf 1}=s_k,\,\,\,\, k=1,2,\dots,n,
$$
i.e., $\{s_1, s_2, \dots, s_n\}$ is a spin system.
The following theorem is valid.

\begin{theorem} \label{2.2}
Every 2-local inner derivation on $\mathcal{J}Spin_n({\Bbb F})$ is a derivation.
\end{theorem}

\begin{proof}
Now, let $a$, $b$ be elements in $\mathcal{J}Spin_n({\Bbb F})$ and
$D_{a,b}(x)=a(bx)-b(ax)$, $x\in \mathcal{J}Spin_n({\Bbb F})$. Then $D_{a,b}$ is
a derivation. We prove that the $(n+1)\times (n+1)$ matrix, which defines the linear operator
$D_{a,b}$ is skew-symmetric.

Indeed, let
$$
a=a^0{\bf 1}+a^1s_1+a^2s_2+\dots+a^ns_n, a^0, a^1, a^2,\dots, a^n\in {\Bbb R},
$$
$$
b=b^0{\bf 1}+b^1s_1+b^2s_2+\dots+b^ns_n, b^0, b^1, b^2,\dots, b^n\in {\Bbb R},
$$
$$
x=x^0{\bf 1}+x^1s_1+x^2s_2+\dots+x^ns_n, x^0, x^1, x^2,\dots, x^n\in {\Bbb R}.
$$
Then
$$
bx=(b^0{\bf 1}+b^1s_1+b^2s_2+\dots+b^ns_n)(x^0{\bf 1}+x^1s_1+x^2s_2+\dots+x^ns_n)
$$
$$
=b^0x^0{\bf 1}+\sum_{k=1}^nb^0x^ks_k+(\sum_{k=1}^{n}b^kx^k){\bf 1}+\sum_{k=1}^nb^kx^0s_k,
$$
$$
a(bx)=(a^0{\bf 1}+a^1s_1+a^2s_2+\dots+a^ns_n)(bx)=
$$
$$
a^0b^0x^0{\bf 1}+\sum_{k=1}^na^0[b^0x^k+b^kx^0]s_k+\sum_{k=1}^{n}a^0b^kx^k{\bf 1}
+\sum_{k=1}^na^kb^0x^0s_k
$$
$$
+\sum_{k=1}^na^k[b^0x^k+b^kx^0]{\bf 1}+\sum_{m=1}^{n}\sum_{k=1}^{n}a^mb^kx^ks_m.
$$
Similarly,
$$
b(ax)=(b^0{\bf 1}+b^1s_1+b^2s_2+\dots+b^ns_n)(ax)=
$$
$$
b^0a^0x^0{\bf 1}+\sum_{k=1}^nb^0[a^0x^k+a^kx^0]s_k+\sum_{k=1}^{n}b^0a^kx^k{\bf 1}
$$
$$
+\sum_{k=1}^nb^ka^0x^0s_k+\sum_{k=1}^nb^k[a^0x^k+a^kx^0]{\bf 1}+\sum_{m=1}^{n}\sum_{k=1}^{n}b^ma^kx^ks_m.
$$
Hence,
$$
D_{a,b}(x)=\sum_{m=1}^n\sum_{k=1, k\neq m}^n(a^mb^k-a^kb^m)x^ks^m.
$$
This sum gives us a skew-symmetric matrix with respect to the basis $\mathcal{V}$.
Therefore the linear operator $D_{a,b}$ is generated by a skew-symmetric matrix with respect to the basis
$\mathcal{V}$.

Now, every inner derivation on $\mathcal{J}Spin_n({\Bbb F})$ is of the following form:
$$
D(x)=\sum_{k=1}^m D_{a_k,b_k}(x)=\sum_{k=1}^m(a_k\cdot (b_k\cdot x)-b_k\cdot (a_k\cdot x)), x\in \mathcal{J}Spin_n({\Bbb F}),
$$
where $a_1$, $a_2$, $\dots,$ $a_m$, $b_1$, $b_2$, $\dots,$ $b_m\in \mathcal{J}Spin_n({\Bbb F})$.
In this case the inner derivation $D$ is also generated by a skew-symmetric matrix with respect to the basis
$\mathcal{V}$.

Let $\Delta$ be a 2-local inner derivation on $\mathcal{J}Spin_n({\Bbb F})$.
Then for every pair of elements
$x$, $y$ in $\mathcal{J}Spin_n({\Bbb F})$ there exists an inner derivation $D$ on $\mathcal{J}Spin_n({\Bbb F})$
such that $\Delta(x)=D(x)$, $\Delta(y)=D(y)$. By our last statement $D$ is generated by a
skew-symmetric matrix with respect to the basis $\mathcal{V}$.
Hence, $D$ is $\Lambda$-symmetric with respect to the basis
$\mathcal{V}$, where $\Lambda=((-1){{\rm sgn}(i-j)})_{i,j=1,\dots,n}$.
Therefore,
since $\Delta$ is a 2-local linear map on $\mathcal{J}Spin_n({\Bbb F})$ we have that $\Delta$ is linear
by Theorem \ref{0.2}. Hence $\Delta$ is a derivation by Theorem \ref{2.1}.
The proof is complete.
\end{proof}

Let  $M_n(\mathcal{F})$ be a matrix ring over a field $\mathcal{F}$, $n>1$, i.e., consisting of matrices
$$
\left[%
\begin{array}{cccc}
a^{1,1} & a^{1,2} & \cdots & a^{1,n}\\
a^{2,1} & a^{2,2} & \cdots & a^{2,n}\\
\vdots & \vdots & \ddots & \vdots\\
a^{n,1} & a^{n,2} & \cdots & a^{n,n}\\
\end{array}%
\right],
a^{i,j}\in \mathcal{F}, i,j=1,2,\dots ,n.
$$
Let $\{e_{i,j}\}_{i,j=1}^n$ be the set of matrix units in
$M_n(\mathcal{F})$, i.e., $e_{i,j}$ is the matrix with components
$a^{i,j}={\bf 1}$ and $a^{k,l}={\bf 0}$ if $(i,j)\neq(k,l)$, where
${\bf 1}$ is the identity element, ${\bf 0}$ is the zero element of the field
$\mathcal{F}$, and a matrix $a\in M_n(\mathcal{F})$ is written as $a=\sum_{k,l=1}^n
a^{k,l}e_{k,l}$, where $a^{k,l}\in \mathcal{F}$ for $k,l=1,2,\dots, n$,
or as $a=\sum_{k,l=1}^n a_{k,l}$, where $a_{k,l}=a^{k,l}e_{k,l}$ for $k,l=1,2,\dots, n$.

For $i\in\{1,2,\dots, n\}$,let $e_i$ be the vector $(a^1,a^2,\dots,a^n)$ in $\mathcal{F}^n$ with the components
$a^i={\bf 1}$ and $a^k={\bf 0}$, $k\neq i$.

We will use the following lemma in the sequel.

\begin{lemma} \label{2.5}
Let $M_n(\mathcal{F})$ be the associative algebra of $n\times n$ matrices over a field $\mathcal{F}$,
and let $d\in M_n(\mathcal{F})$ be a skew-symmetric matrix such that $D(x)=D_d(x)=dx-xd$, for any $x\in M_n(\mathcal{F})$,
i.e., $D$ is an inner derivation. Then the $n^2\times n^2$ matrix $A$, with respect to the basis
$$
\mathcal{V}=\{e_{1,1},e_{2,1},\dots,e_{n,1},e_{1,2},e_{2,2},\dots,e_{n,2},\dots,e_{1,n},e_{2,n},\dots,e_{n,n}\}. \eqno{(2.1)}
$$
of the derivation $D$, is skew-symmetric.
\end{lemma}

\begin{proof}
The matrices of the multiplications
$$
L(d):x\mapsto dx \text{ and } R(d): x\mapsto xd,\,\,\,x\in M_n(\mathcal{F})
$$
with respect to the basis (2.1) can be expressed as $I \otimes d$ and $d^T \otimes I$, respectively,
in terms of the usual tensor product of matrices, the transposition $\cdot^T$ and the identity matrix $I$.
Hence, the matrix of $D$ with respect to the basis (2.1) is $I \otimes d-d^T \otimes I$ which is skew-symmetric. The proof is complete.
\end{proof}

Let ${\Bbb F}$ be a field of characteristic different from $2$. We recall that
a quaternion algebra $\mathcal{Q}({\Bbb F})$ is associative, has an identity element ${\bf 1}$ and two
generators $i$, $j$ with the defining relations
$$
i^2=\lambda {\bf 1}, j^2=\mu {\bf 1}, ij=-ji,
$$
where $\lambda$ and $\mu$ are nonzero elements of ${\Bbb F}$. The elements ${\bf 1}$, $i$,
$j$, $k=ij$ form a basis for $\mathcal{Q}({\Bbb F})$ and one has the multiplication table:
$$
i^2=\lambda {\bf 1}, j^2=\mu {\bf 1}, k^2=-\lambda\mu {\bf 1}, ij=-ji=k,
$$
$$
jk=-kj=-\mu i, ki=-ik=-\lambda j
$$
(together with ${\bf 1}a=a=a{\bf 1}$). We recall that $\mathcal{Q}({\Bbb F})$ has an involution $a\to \bar{a}$ such
that $\bar{a}=\alpha_0{\bf 1}-\alpha_1i-\alpha_2j-\alpha_3k$, if $a=\alpha_0{\bf 1}+\alpha_1i+\alpha_2j+\alpha_3k$.

Similarly, a binarion algebra $\mathcal{B}({\Bbb F})$ is associative, has an identity element ${\bf 1}$ and one
generator $i$ with the defining relation
$$
i^2=\lambda {\bf 1},
$$
where $\lambda$ is a nonzero element of ${\Bbb F}$. The elements ${\bf 1}$, $i$
form a basis for $\mathcal{B}({\Bbb F})$ and one has the multiplication table:
$$
i^2=\lambda {\bf 1}, {\bf 1}i=i=i{\bf 1}.
$$
We recall that $\mathcal{B}({\Bbb F})$ has an involution $a\to \bar{a}$ such
that $\bar{a}=\alpha_0{\bf 1}-\alpha_1i$, if $a=\alpha_0{\bf 1}+\alpha_1i$.
(cf. \cite[Chapter I, \S 5]{NJ})

If $\mathcal{A}$ is a $*$-algebra (i.e., an involutive algebra), in particular, if $\mathcal{A}={\Bbb F}$, $\mathcal{B}({\Bbb F})$,
$\mathcal{Q}({\Bbb F})$, then so is $M_n(\mathcal{A})$, the space of $n\times n$ matrices
with coefficients in $\mathcal{A}$, with $(a_{i,j})^*=(a_{j,i}^*)$. The hermitian, or self-adjoint, part
of $M_n(\mathcal{A})$ is denoted by $H_n(\mathcal{A})$. $H_n(\mathcal{A})$ is a Jordan algebra, with the product defined
by $a\circ b=\frac{1}{2}(ab+ba)$ since $\mathcal{A}$ is associative. A Jordan algebra of the form
$H_n(\mathcal{A})$ is called a Jordan matrix algebra \cite[Chapter I, \S 5]{NJ}.
The following theorem holds.

\begin{theorem} \label{2.6}
Let ${\Bbb F}$ be a field of characteristic different from $2$ and
let $H_n(\mathcal{A})$ be the Jordan algebra of self-adjoint $n\times n$ matrices over
a division $*$-algebra $\mathcal{A}$, $n\geq 3$, where $\mathcal{A}={\Bbb F}$, $\mathcal{B}({\Bbb F})$,
$\mathcal{Q}({\Bbb F})$. Then every 2-local inner derivation on $H_n(\mathcal{A})$ is a derivation.
\end{theorem}

\begin{proof}
If $\mathcal{A}={\Bbb F}$, then by theorem 15 in \cite{AA3} the theorem is valid.

Let $\mathcal{A}=\mathcal{B}({\Bbb F})$ and $\Delta$ be a 2-local inner derivation on $H_n(\mathcal{A})$. We define
a map $\phi:M_n(\mathcal{A})\to M_n(\mathcal{A})$ as follows
$$
\phi(x)=\Delta(x), x\in H_n(\mathcal{A}).
$$
$$
\phi(x)=\Delta(x+x^*), \text{ if } x\in M_n(\mathcal{A}) \text{ and } x\neq x^*.
$$
Then $\phi$ is a 2-local linear map. We prove that $\phi$ is linear.

Let
$$
\mathcal{V}=\{e_{1,1},e_{2,1},\dots,e_{n,1},e_{1,2},e_{2,2},\dots,e_{n,2},\dots,e_{1,n},e_{2,n},\dots,e_{n,n}\}
$$
be the system of matrix units in $M_n(\mathcal{A})$.
Let $x$ be an element in $H_n(\mathcal{A})$ and $a$, $b$ be elements in $H_n(\mathcal{A})$ such that
$$
\phi(x)=D_{a,b}(x)=a(bx)-b(ax)
$$
We prove that the matrix $A$ of the inner derivation $D_{a,b}$ with respect to $\mathcal{V}$ is skew-symmetric.

Indeed, the following equality is valid
$$
D_{a,b}(y)=\frac{1}{4}((ab-ba)y-y(ab-ba)), y\in H_n(\mathcal{A})
$$
with respect to the associative multiplication in $M_n(\mathcal{A})$.
The linear operator, defined as
$$
D_{a,b}(y)=\frac{1}{4}((ab-ba)y-y(ab-ba)), y\in M_n(\mathcal{A}),
$$
is generated by the $n^2\times n^2$ matrix $A$.
Since $ab-ba$ is skew-symmetric we have $A$ is skew-symmetric by Lemma \ref{2.5}.

Let $x$ be an element in $M_n(\mathcal{A})$ such that $x^*\neq x$ and $a$, $b$ be elements in $H_n(\mathcal{A})$ such that
$$
\phi(x)=D_{a,b}(x+x^*)=a(b(x+x^*))-b(a(x+x^*)).
$$
The map
$$
L(y)=D_{a,b}(y+y^*), y\in M_n(\mathcal{A})
$$
is linear. We prove that the matrix $A$ of the map $L(y)$ with respect to $\mathcal{V}$ is skew-symmetric.

Indeed, the following equality is valid
$$
L(y)=\frac{1}{4}((ab-ba)(y+y^*)-(y+y^*)(ab-ba)), y\in M_n(\mathcal{A})
$$
with respect to the associative multiplication in $M_n(\mathcal{A})$.
We take the operator
$$
D_{a,b}^*(y)=\frac{1}{4}((ab-ba)y^*-y^*(ab-ba)), y\in M_n(\mathcal{A}).
$$
We have
$$
(D_{a,b}^*(y))^*=(\frac{1}{4}((ab-ba)y^*-y^*(ab-ba)))^*=
$$
$$
=\frac{1}{4}(y(-ab+ba)-(-ab+ba)y)=\frac{1}{4}((ab-ba)y-y(ab-ba))=D_{a,b}(y),
$$
i.e., the map $D_{a,b}^*$ is adjoint to $D_{a,b}$. Hence,
the matrix of the map $D_{a,b}^*$ with respect to $\mathcal{V}$ is $A^*$ and skew-symmetric.
Since $L(y)=D_{a,b}(y)+D_{a,b}^*(y)$ we have the matrix of the map $L(y)$ with respect
to $\mathcal{V}$ is $A+A^*$ and skew-symmetric.

Now, every inner derivation on $H_n(\mathcal{A})$ is of the following kind
$$
D(x)=\sum_{k=1}^m D_{a_k,b_k}(x)=\sum_{k=1}^m(a_k\cdot (b_k\cdot x)-b_k\cdot (a_k\cdot x)), x\in H_n(\mathcal{A}),
$$
where $a_1$, $a_2$, $\dots,$ $a_m$, $b_1$, $b_2$, $\dots,$ $b_m\in H_n(\mathcal{A})$.
In this case the inner derivation $D$ is also generated by a skew-symmetric matrix
with respect to the system $\mathcal{V}$. It is a sum of $m$ skew-symmetric matrices.
Similarly, the matrix $A$ of the map
$$
L(y)=\sum_{k=1}^m D_{a_k,b_k}(y+y^*), y\in M_n(\mathcal{A})
$$
with respect to $\mathcal{V}$ is also skew-symmetric.
Hence, by Theorem \ref{0.2}, $\phi$ is linear.
Therefore, $\Delta$ is also linear, since it is homogeneous, and by Theorem \ref{2.1}, $\Delta$ is
a derivation.

Let $\mathcal{A}=\mathcal{Q}({\Bbb F})$.
Then $\mathcal{Q}({\Bbb F})$ is embedded in $M_4({\Bbb F})$ and we can take
a matrix $I$ in $M_4({\Bbb F})$ such that $I^2=-E$, where $E$ is the unit matrix in
$M_4({\Bbb F})$, and the sum
$$
\mathcal{Q}({\Bbb F})+I\mathcal{Q}({\Bbb F})
$$
is isomorphic to $M_2(\mathcal{B}({\Bbb F}))$. Hence
$M_n(\mathcal{A}+i\mathcal{A})=M_n(\mathcal{Q}({\Bbb F})+I\mathcal{Q}({\Bbb F}))$ is isomorphic to
$M_{2n}(\mathcal{B}({\Bbb F}))$. Hence we can apply Lemma \ref{2.5} to
the algebra $M_n(\mathcal{A}+i\mathcal{A})$.

Let $\Delta$ be a 2-local inner derivation on $H_n(\mathcal{A})$. We define
a map $\phi:M_n(\mathcal{A})\to M_n(\mathcal{A})$ as follows
$$
\phi(x)=\Delta(x), x\in H_n(\mathcal{A}).
$$
$$
\phi(x)=\Delta(x+x^*), \text{ if } x\in M_n(\mathcal{A}) \text{ and } x\neq x^*.
$$
Then $\phi$ is a 2-local linear map.

We define a map $\Phi:M_n(\mathcal{A}+i\mathcal{A})\to M_n(\mathcal{A}+i\mathcal{A})$ as follows
$$
\Phi(x)=\phi(x), x\in M_n(\mathcal{A}).
$$
$$
\Phi(x)=\phi(a), \text{ if } x=a+ib, a, b\in M_n(\mathcal{A}).
$$
Then $\Phi$ is a 2-local linear map. Similarly to the previous part of the proof we prove that
for every $x\in M_n(\mathcal{A})$ the matrix $A$ with respect to
the basis of $M_n(\mathcal{A}+i\mathcal{A})$, corresponding to the canonical basis
of the algebra $M_{2n}(\mathcal{B}({\Bbb F}))$ over the field $\mathcal{B}({\Bbb F})$,
of the linear operator $L$ such
that $\Phi(x)=L(x)$ is skew-symmetric by Lemma \ref{2.5}.
Hence by Theorem \ref{0.2} $\Phi$ is linear.
Therefore $\Delta$ is also linear, and by Theorem \ref{2.1} $\Delta$ is
a derivation. The proof is complete.
\end{proof}

Let ${\Bbb F}$ be a field of characteristic different from $2$ and let
$\mathcal{O}({\Bbb F})=\mathcal{Q}({\Bbb F})\oplus \mathcal{Q}({\Bbb F})$
the vector space direct sum of two copies of the quaternion algebra $\mathcal{Q}({\Bbb F})$.
We write the elements of $\mathcal{O}({\Bbb F})$ as $(a,b)$ where $a$, $b\in\mathcal{Q}({\Bbb F})$.
Let $\nu$ be a nonzero element of $\mathcal{Q}({\Bbb F})$ and define a product in $\mathcal{O}({\Bbb F})$
by
$$
(a,b)(c,d)=(ac+\nu db, da+b\bar{c}),
$$
where $a$, $b$, $c$, $d\in \mathcal{Q}({\Bbb F})$. Since $\mathcal{Q}({\Bbb F})$ is an algebra and
$a\to\bar{a}$ is linear it is clear that the product defined in $\mathcal{O}({\Bbb F})$ is bilinear,
so $\mathcal{O}({\Bbb F})$ is an algebra. $\mathcal{O}({\Bbb F})$ is called the algebra of octonions
(Cayley algebra, Cayley-Dickson algebra, Cayley-Graves algebra), defined by $\mathcal{Q}({\Bbb F})$
and $\nu$.

It is clear that $(1,0)$ is an identity element for $\mathcal{O}({\Bbb F})$ and that the subset
of elements $(a,0)$ is a subalgebra of $\mathcal{O}({\Bbb F})$ isomorphic under $(a,0)\to a$ to
$\mathcal{Q}({\Bbb F})$. We identify $\mathcal{Q}({\Bbb F})$ with the corresponding subalgebra of
$\mathcal{O}({\Bbb F})$ and write $a$ for the element $(a,0)$ of $\mathcal{O}({\Bbb F})$. Also
we set $l=(0,1)$. Then $bl=(b,0)(0,1)=(0,b)=l\bar{b}$. Hence every element of
$\mathcal{O}({\Bbb F})$ can be written in one and only one way as $a+bl$, $a$, $b$ in
$\mathcal{Q}({\Bbb F})$. For $x=a+bl$, $a$, $b$ in $\mathcal{Q}({\Bbb F})$ we define
$\bar{x}=\bar{a}-bl$. Then $x\to \bar{x}$ is an involution in $\mathcal{O}({\Bbb F})$.
(cf. \cite[Chapter I, \S 5]{NJ}, \cite[\S 2.13]{KMcC})

Let $\mathcal{H}_3(\mathcal{O}({\Bbb F}))$ be the vector space of all self-adjoint
$3\times 3$ matrices with octonion entries. Then, by theorem 4 of chapter I and theorem 5 of chapter IV in \cite{NJ},
$\mathcal{H}_3(\mathcal{O}({\Bbb F}))$
is a Jordan algebra, with the product defined by $x\circ y=\frac{1}{2}(xy+yx)$.

Now, let
$$
t_{\alpha,\alpha}=e_{\alpha,\alpha}, t_{\alpha,\beta}=e_{\alpha,\beta}+e_{\beta,\alpha}, v_{\alpha,\beta}=e_{\alpha,\beta}-e_{\beta,\alpha},
$$
$$
I_{\alpha,\beta}=iv_{\alpha,\beta}, J_{\alpha,\beta}=jv_{\alpha,\beta}, K_{\alpha,\beta}=kv_{\alpha,\beta},
$$
$$
LI_{\alpha,\beta}=liv_{\alpha,\beta}, LJ_{\alpha,\beta}=ljv_{\alpha,\beta}, LK_{\alpha,\beta}=lkv_{\alpha,\beta},
$$
$$
L_{\alpha,\beta}=lv_{\alpha,\beta}, \alpha<\beta, \alpha,\beta\in\{1,2,3\}.
$$
Then, the family
$$
\mathcal{V}=\{t_{\alpha,\alpha}, t_{\alpha,\beta}, I_{\alpha,\beta}, J_{\alpha,\beta}, K_{\alpha,\beta}, LI_{\alpha,\beta}, LJ_{\alpha,\beta}, LK_{\alpha,\beta}, L_{\alpha,\beta}, \alpha<\beta, \alpha,\beta\in\{1,2,3\}\}
$$
is a basis of the vector space $\mathcal{H}_3(\mathcal{O}({\Bbb F}))$, i.e., $\mathcal{H}_3(\mathcal{O}({\Bbb F}))$ is 27-dimensional,
and this basis has an appropriate table of multiplication. Suppose that the field ${\Bbb F}$ is algebraically closed.
Then the following theorem is valid.

\begin{theorem} \label{2.7}
Every 2-local inner derivation on $\mathcal{H}_3(\mathcal{O}({\Bbb F}))$ is a derivation.
\end{theorem}

\begin{proof}
Firstly, we prove that, for any $a$, $b$ in $\mathcal{H}_3(\mathcal{O}({\Bbb F}))$, the $27\times 27$ matrix
with respect to $\mathcal{V}$, which generates the linear operator
$D_{a,b}(x)=a(bx)-b(ax)$, $x\in \mathcal{H}_3(\mathcal{O}({\Bbb F}))$, is skew-symmetric.

Let $\mathcal{J}(a,b)$ be the Jordan subalgebra of $\mathcal{H}_3(\mathcal{O}({\Bbb F}))$, generated by the elements $a$, $b$.
By the Shirshov-Cohn theorem the Jordan algebra $\mathcal{J}(a,b)$ is special.
Hence, by \cite[Chapter V, \S 6, Corollary 2]{NJ}, \cite[Chapter V, \S 5, Corollary 2]{NJ},
\cite[Chapter V, \S 5, Second Structure theorem]{NJ} $\mathcal{J}(a,b)$ is isomorphic to the direct sum of
the Jordan algebras indicated in Theorems \ref{2.2}, \ref{2.6} and the Jordan algebra
${\Bbb F}$ (cf. \cite[\S 2.13]{KMcC}). Therefore, by Theorems \ref{2.2}, \ref{2.6},
the matrix of the derivation
$$
D_{a,b}(x)=a(bx)-b(ax), x\in \mathcal{J}(a,b)
$$
is skew-symmetric with respect to the basis
$\mathcal{B}=\{b_1,b_2,...,b_k\}$ of $\mathcal{J}(a,b)$ corresponding to the standard basis of this direct sum, and
there exists a basis
$\mathcal{B}_1=\{b_1^1,b_2^1,...,b_{m_1}^1,b_1^2,b_2^2,...,b_{m_2}^2,...,b_1^{k+1},b_2^{k+1},...,b_{m_{k+1}}^{k+1}\}$ in $\mathcal{H}_3(\mathcal{O}({\Bbb F}))$
with a table of multiplication similarly to the table of multiplication of the basis $\mathcal{V}$ such that
$$
b_\alpha=\lambda_1^\alpha b_1^\alpha+\lambda_2^\alpha b_2^\alpha+...+\lambda_{m_1}^\alpha b_{m_1}^\alpha, \alpha=1,2,...,k.
$$
Then the matrix of the derivation $D_{a,b}$ with respect to the basis $\mathcal{B}_1$ is also skew-symmetric.
Since the elements $a$, $b$ are chosen arbitrarily, we have the matrix of the derivation $D_{a,b}$ with respect
to the basis $\mathcal{V}$ is skew-symmetric to.

Now, every inner derivation on $\mathcal{H}_3(\mathcal{O}({\Bbb F}))$ is of the following kind
$$
D(x)=\sum_{k=1}^m D_{a_k,b_k}(x)=\sum_{k=1}^m(a_k\cdot (b_k\cdot x)-b_k\cdot (a_k\cdot x)), x\in H_n(\mathcal{A}),
$$
where $a_1$, $a_2$, $\dots,$ $a_m$, $b_1$, $b_2$, $\dots,$ $b_m\in H_n(\mathcal{A})$.
In this case the inner derivation $D$ is also generated by a skew-symmetric matrix
with respect to the system $\mathcal{V}$. It is a sum of $m$ skew-symmetric matrices.
Hence, every 2-local inner derivation on $\mathcal{H}_3(\mathcal{O}({\Bbb F}))$ is linear
by Theorem \ref{0.2} and a derivation by Theorem \ref{2.1}.
This ends the proof.
\end{proof}

By \cite[Chapter V, \S 6, Corollary 2]{NJ}, \cite[Chapter V, \S 5, Corollary 2]{NJ},
\cite[Chapter V, \S 5, Second Structure theorem]{NJ}, every finite dimensional semisimple Jordan algebra over an
algebraically closed field ${\Bbb F}$ of characteristic different from $2$ is isomorphic to the direct sum of
the Jordan algebras indicated in Theorems \ref{2.2}, \ref{2.6}, \ref{2.7} and the Jordan algebra
${\Bbb F}$. Every 2-local inner derivation on
a finite dimensional semisimple Jordan algebra is a direct sum of 2-local inner derivations
on its subalgebras isomorphic to Jordan algebras in the direct sum. Therefore,
the following theorem is valid.

\begin{theorem} \label{2.8}
Every 2-local inner derivation on a finite dimensional semisimple Jordan algebra over an
algebraically closed field of characteristic different from $2$ is a derivation.
\end{theorem}

\section{On 2-Local automorphisms of finite dimensional semisimple Jordan algebras}

Recall that a 2-local automorphism is defined as follows:
given a Jordan algebra $\mathcal{J}$, a map $\Delta : \mathcal{J} \to \mathcal{J}$ (not linear in
general) is called a 2-local automorphism if, for every $x$, $y\in
\mathcal{J}$, there exists an automorphism $\Phi_{x,y} : \mathcal{J}\to \mathcal{J}$ such that
$\Delta(x)=\Phi_{x,y}(x)$ and $\Delta(y)=\Phi_{x,y}(y)$.

\begin{theorem} \label{3.1}
Let $\mathcal{J}$ be a finite dimensional Jordan algebra over a field $\mathcal{F}$ of characteristic different from $2$ and let $\Delta$ be a 2-local
automorphism on $\mathcal{J}$, $\Lambda=(\lambda_{i,j})_{i,j=1,\dots,n}$ be an $n\times n$ matrix with non-zero entries of
$\mathcal{F}$, $\mathcal{V}$ be a basis of $\mathcal{J}$.
Assume that, for every $x$, $y$ in $\mathcal{J}$, there is an automorphism $\Phi_{x,y}$, which is $\Lambda$-symmetric with
respect to $\mathcal{V}$, and such that
$$
\Delta(x)=\Phi_{x,y}(x),\,\,\, \Delta(y)=\Phi_{x,y}(y).
$$
Then $\Delta$ is an automorphism.
\end{theorem}

\begin{proof}
By the definition, $\Delta$ is a $\Lambda$-symmetric 2-local linear map. Hence, by Theorem \ref{0.2},
$\Delta$ is a linear operator. Now applying the definition of 2-local automorphism to the elements  $x$, $x^2\in
\mathcal{A}$,  we obtain that  $\Delta$
is a Jordan automorphism, i.e., $\Delta(x^2)=\Delta(x)\Delta(x)$, $x\in \mathcal{J}$.
The proof is complete.
\end{proof}

\begin{definition}
Let $\mathcal{J}$ be a unital Jordan algebra. An automorphism $\Phi$ of $\mathcal{J}$ is called a 1-automorphism
if $\Phi(x)= U_s(x) = 2s(sx)-s^2x = 2s(sx)-x$, $x\in\mathcal{J}$, for some symmetry $s\in \mathcal{J}$, i.e., $s^2=1$.
\end{definition}

\begin{theorem} \label{3.2}
For every 1-automorphism $\Phi$ on $\mathcal{J}Spin_n({\Bbb F})$, the $(n+1)\times (n+1)$ matrix $A$,
which generates the linear operator $\Phi$, is symmetric with respect to the basis $\mathcal{V}=\{1,s_1, s_2, \dots, s_n\}$,
where $\{s_1, s_2, \dots, s_n\}$ is a spin system.
\end{theorem}

\begin{proof}
Indeed, let
$$
s=a^0{\bf 1}+a^1s_1+a^2s_2+\dots+a^ns_n, a^0, a^1, a^2,\dots, a^n\in {\Bbb R},
$$
$$
x=x^0{\bf 1}+x^1s_1+x^2s_2+\dots+x^ns_n, x^0, x^1, x^2,\dots, x^n\in {\Bbb R}.
$$
Then
$$
sx=(a^0{\bf 1}+a^1s_1+a^2s_2+\dots+a^ns_n)(x^0{\bf 1}+x^1s_1+x^2s_2+\dots+x^ns_n)
$$
$$
=a^0x^0{\bf 1}+\sum_{k=1}^na^0x^ks_k+(\sum_{k=1}^{n}a^kx^k){\bf 1}
+\sum_{k=1}^na^kx^0s_k
$$
$$
=a^0x^0{\bf 1}+\sum_{k=1}^n[a^0x^k+a^kx^0]s_k+(\sum_{k=1}^{n}a^kx^k){\bf 1},
$$
$$
s(sx)=(a^0{\bf 1}+a^1s_1+a^2s_2+\dots+a^ns_n)(sx)=
$$
$$
(a^0)^2x^0{\bf 1}+\sum_{k=1}^na^0[a^0x^k+a^kx^0]s_k+\sum_{k=1}^{n}a^0a^kx^k
$$
$$
+\sum_{k=1}^na^ka^0x^0s_k+\sum_{k=1}^na^k[a^0x^k+a^kx^0]{\bf 1}+\sum_{m=1}^{n}\sum_{k=1}^{n}a^ma^kx^ks_m.
$$
Hence,
$$
U_s(x)=2s(sx)-x=
$$
$$
2(a^0)^2x^0{\bf 1}+\sum_{k=1}^n2a^0[a^0x^k+a^kx^0]s_k+\sum_{k=1}^{n}2a^0a^kx^k{\bf 1}
$$
$$
+\sum_{k=1}^n2a^ka^0x^0s_k+\sum_{k=1}^n2a^k[a^0x^k+a^kx^0]{\bf 1}+\sum_{m=1}^{n}\sum_{k=1}^{n}2a^ma^kx^ks_m
$$
$$
-(x^0{\bf 1}+x^1s_1+x^2s_2+\dots+x^ns_n)=
$$
$$
(2\sum_{k=0}^n(a^k)^2-1)x^0{\bf 1}+4\sum_{k=1}^{n}a^0a^kx^k{\bf 1}+4\sum_{k=1}^na^0a^kx^0s_k
$$
$$
+\sum_{k=1}^n(2(a^0a^0+a^ka^k)-1)x^ks_k+2\sum_{m=1}^{n}\sum_{k=1,k\neq m}^{n}a^ma^kx^ks_m.
$$

This sum gives us a symmetric matrix with respect to the basis $\mathcal{V}$.
Therefore the linear operator $U_s$ is generated by a symmetric matrix with respect to the basis
$\mathcal{V}$. The proof is complete.
\end{proof}

\begin{definition}
A map $\Delta$ on $\mathcal{J}$ is called a 2-local 1-automorphism, if
for every $x$, $y\in \mathcal{J}$ there exists a symmetry $s$ in $\mathcal{J}$ such that
$\Delta(x)=U_s(x)$, $\Delta(y)=U_s(y)$.
\end{definition}

The following theorem is valid by Theorems \ref{3.1} and \ref{3.2}.

\begin{theorem} \label{3.3}
Every 2-local 1-automorphism of $\mathcal{J}Spin_n({\Bbb F})$ is an automorphism.
\end{theorem}

We will use the following lemma in the sequel.

\begin{lemma} \label{3.4}
Let $M_n(\mathcal{F})$ be the associative algebra of $n\times n$ matrices over a field $\mathcal{F}$
and let $u$, $v$ be symmetric matrices in  $M_n(\mathcal{F})$. Consider the linear map  $\phi_{u,v}(x)=uxv$,  $x\in M_n(\mathcal {F})$.
Then the $n^2\times n^2$ matrix $A$ of the linear operator $\phi_{u,v}$ with respect to basis (2.1) is symmetric.
\end{lemma}

\begin{proof}
The matrices of the multiplications
$$
L(u):x\mapsto ux \text{ and } R(v): x\mapsto xv,\,\,\,x\in M_n(\mathcal{F})
$$
with respect to the basis (2.1) can be expressed as $I \otimes u$ and $v^T \otimes I$, respectively.
Hence, the matrix of the linear operator $\phi_{u,v}$, with respect to basis (2.1),
is the product $(I \otimes u)(v^T \otimes I)$ of the matrices $I \otimes u$ and $v^T \otimes I$,
which is symmetric. The proof is complete.
\end{proof}

\begin{theorem} \label{3.5}
Let ${\Bbb F}$ be a field of characteristic different from $2$ and
let $H_n(\mathcal{A})$ be the Jordan algebra of self-adjoint $n\times n$ matrices over
a division $*$-algebra $\mathcal{A}$, $n\geq 3$, where $\mathcal{A}={\Bbb F}$, $\mathcal{B}({\Bbb F})$,
$\mathcal{Q}({\Bbb F})$. Then every 2-local 1-automorphism on $H_n(\mathcal{A})$ is an automorphism.
\end{theorem}

\begin{proof}
Let $\mathcal{A}={\Bbb F}$ or $\mathcal{B}({\Bbb F})$ and
let $s$ be an arbitrary symmetry in $H_n(\mathcal{A})$ and
$$
\mathcal{V}=\{e_{1,1},e_{2,1},\dots,e_{n,1},e_{1,2},e_{2,2},\dots,e_{n,2},\dots,e_{1,n},e_{2,n},\dots,e_{n,n}\}
$$
be the standard basis of $M_n(\mathcal{A})$. Then
$U_s(x)$, $x\in H_n(\mathcal{A})$, is a 1-automorphism of $H_n(\mathcal{A})$
and, with respect to the associative multiplication in $M_n(\mathcal{A})$,
$$
U_s(x)=sxs, x\in M_n(\mathcal{A}).
$$
is a linear operator on $M_n(\mathcal{A})$ and, by Lemma \ref{3.4} this linear operator
is generated by a symmetric $n^2\times n^2$ matrix $A$ with respect to the basis $\mathcal{V}$, since $s$ is symmetric.

Let $\Delta$ be a 2-local 1-automorphism of $H_n(\mathcal{A})$. Then, for every pair of elements
$x$, $y$ in $H_n(\mathcal{A})$, there exists a 1-automorphism $\Phi$ of $H_n(\mathcal{A})$
such that $\Delta(x)=\Phi(x)$, $\Delta(y)=\Phi(y)$. By the statement proved above, $\Phi$ is generated by a
symmetric matrix with respect to the basis $\mathcal{V}$.

We define a map $\phi:M_n(\mathcal{A})\to M_n(\mathcal{A})$ as follows
$$
\phi(x)=\Delta(x),\,\,\, x\in H_n(\mathcal{A}).
$$
$$
\phi(x)=\Delta(x+x^*), \text{ if } x\in M_n(\mathcal{A}) \text{ and } x\neq x^*.
$$
Then $\phi$ is a 2-local linear map. We prove that $\phi$ is linear.

Let $x$ be an element in $M_n(\mathcal{A})$ such that $x^*\neq x$. Then
$$
\phi(x)=\Delta(x+x^*)=s(x+x^*)s,
$$
for some symmetry $s$ in $H_n(\mathcal{A})$, and
the map
$$
L(y)=s(y+y^*)s,\,\,\, y\in M_n(\mathcal{A})
$$
is linear. We prove that the matrix $A$ of the map $L(y)$ with respect to $\mathcal{V}$ is symmetric.

Indeed, we take the operator
$$
U_s^*(y)=sy^*s,\,\,\, y\in M_n(\mathcal{A}).
$$
We have
$$
U_s^*(y)=(U_s(y))^*,
$$
i.e., the map $U_s^*$ is adjoint to $U_s$. Hence,
the matrix of the map $U_s^*$ with respect to $\mathcal{V}$ is $A^*$ and symmetric.
Since $L(y)=U_s(y)+U_s^*(y)$, we have the matrix of the map $L(y)$ with respect
to $\mathcal{V}$ is $A+A^*$ and symmetric.
Hence, by Theorem \ref{0.2}, $\phi$ is linear.
Therefore $\Delta$ is linear, since it is homogenous, and, by Theorem \ref{3.1}, $\Delta$ is
an automorphism.

Now let $\mathcal{A}=\mathcal{Q}({\Bbb F})$.
Since $M_n(\mathcal{A}+i\mathcal{A})=M_n(\mathcal{Q}({\Bbb F})+I\mathcal{Q}({\Bbb F}))$ is isomorphic to
$M_{2n}(\mathcal{B}({\Bbb F}))$, we can apply Lemma \ref{3.4} to $M_n(\mathcal{A}+i\mathcal{A})$.

Let $\Delta$ be a 2-local 1-automorphism on $H_n(\mathcal{A})$. We define
a map $\phi:M_n(\mathcal{A})\to M_n(\mathcal{A})$ as follows
$$
\phi(x)=\Delta(x),\,\,\, x\in H_n(\mathcal{A}).
$$
$$
\phi(x)=\Delta(x+x^*), \text{ if } x\in M_n(\mathcal{A}) \text{ and } x\neq x^*.
$$
Then $\phi$ is a 2-local linear map.

In turn we define a map $\Phi:M_n(\mathcal{A}+i\mathcal{A})\to M_n(\mathcal{A}+i\mathcal{A})$ as follows
$$
\Phi(x)=\phi(x),\,\,\, x\in M_n(\mathcal{A}),
$$
$$
\Phi(x)=\phi(a), \text{ if } x=a+ib,\,\,\, a, b\in M_n(\mathcal{A}).
$$
Then $\phi$ is a 2-local linear map. Similarly to the previous part of the proof, we prove that,
for every $x\in M_n(\mathcal{A})$, the matrix $A$ with respect to $\mathcal{V}$ of the linear operator $L$ such
that $\Phi(x)=L(x)$ is symmetric, by Lemma \ref{3.4}.
The proof is complete.
\end{proof}

The following theorem is proved similarly to the proof of Theorem \ref{2.7},
using Theorem \ref{0.2}.

\begin{theorem} \label{3.6}
Every 2-local 1-automorphism of $H_3({\Bbb O})$ is an automorphism.
\end{theorem}

Every 2-local 1-automorphism on a finite dimensional semisimple Jordan algebra $\mathcal{J}$
is a direct sum of 2-local 1-automorphisms
on its subalgebras, the direct sum of which is $\mathcal{J}$ and isomorphic to Jordan algebras indicated in
Theorems \ref{3.3}, \ref{3.5}, \ref{3.6} and the Jordan algebra ${\Bbb F}$. Therefore,
the following theorem is valid.

\begin{theorem} \label{3.7}
Every 2-local 1-automorphism on a finite dimensional semisimple Jordan algebra over an algebraically closed field
of characteristic different from $2$ is an automorphism.
\end{theorem}

\section{Another proof of the theorems in the case of a formally real Jordan algebra} \label{formally real}

A real Jordan algebra $\mathcal{A}$ is called formally real if, for any finite set of elements $a_i\in \mathcal{A}$, whenever
$\sum a_i^2=0$, then $a_i=0$ for each $i$.

We recall that a finite dimensional formally real Jordan algebra $\mathcal{A}$ has a symmetric positive
definite bilinear form $B : \mathcal{A}\times \mathcal{A} \to {\Bbb R}$ which is associative, i.e., such that $B(xy,z)=B(x,yz)$
for all $x$, $y$, $z\in \mathcal{A}$ by theorem 6 in \cite{RDS}. If, for $x\in \mathcal{A}$, we define
$$
L(x): \mathcal{A} \to \mathcal{A},\,\,\, y \to xy,
$$
then the associativity of $B$ is equivalent to
$$
B(L(y)x,z)=B(x,L(y)z)\,\,\, x, y, z\in \mathcal{A}
$$
that is $L(y)$ is symmetric with respect to $B$. Let $\mathcal{V}=(v_1,\dots,v_n)$ be an orthonormal basis of $\mathcal{A}$.
Then the matrix $(B(v_i,v_j))_{i,j=1,\dots,n}$ is the identity, and if $X(y)$ is the matrix of $L(y)$ with respect
to $\mathcal{V}$, then $X(y)$ is symmetric.

For every $a$, $b$ in $\mathcal{A}$, the commutator $[L(a),L(b)]:=L(a)L(b)-L(b)L(a)$ is a derivation of
$\mathcal{A}$. A derivation $D$ of the form $D=\sum_i [L(a_i),L(b_i)]$ is called an inner derivation of $\mathcal{A}$. Every
derivation of a formally real finite dimensional Jordan algebra is inner \cite{UH}.

In the following lemma, $\mathcal{A}$ is a finite dimensional formally real Jordan algebra with a fixed associative
positive definite symmetric bilinear form $B$ and $\mathcal{V}$ an orthonormal basis of $\mathcal{A}$ with respect to $B$.

\begin{lemma} \label{5.3}
Let $a$, $b$ be elements of $\mathcal{A}$. Then

1. $[L(a),L(b)]$ is skew-symmetric relative to $B$. In particular the matrix of $[L(a),L(b)]$ is
skew-symmetric;

2. If $[L(a),L(b)]=0$ then $L(a)L(b)$ is symmetric relative to $B$. In particular the matrix of
$L(a)L(b)$ is symmetric;

3. $L(a)^2$ is symmetric relative to $B$. In particular the matrix of $L(a)^2$ is symmetric.
\end{lemma}

\begin{proof}
We have
$$
B(L(a)L(b)x,y)=B(L(b)x,L(a)y)=B(x,L(b)L(a)y),
$$
for every $x$, $y\in \mathcal{A}$. Therefore
$$
B([L(a),L(b)]x,y)=-B(x,[L(a),L(b)]y)
$$
so that $[L(a),L(b)]$ is skew-symmetric relative to $B$, and its matrix relative to $\mathcal{V}$ is skew-symmetric.
If $[L(a),L(b)]=0$, then
$$
B(L(a)L(b)x,y)=B(x,L(b)L(a)y)= B(x,L(a)L(b)y)
$$
so that $L(a)L(b)$ is symmetric relative to $B$, and its matrix relative to $\mathcal{V}$ is symmetric. In particular
$L(a)^2$ is symmetric.
\end{proof}

\begin{proposition} \label{5.4}
Let $\mathcal{A}$ be a finite dimensional formally real Jordan algebra. Then every derivation of
$\mathcal{A}$ is skew-symmetric, and its matrix is skew-symmetric.
\end{proposition}

\begin{proof}
This proposition follows from item 1 of Lemma \ref{5.3} and the fact that every derivation of $\mathcal{A}$ is inner.
\end{proof}

\begin{theorem} \label{5.5}
Every 2-local derivation of a finite dimensional formally real Jordan algebra $\mathcal{A}$ is
a derivation.
\end{theorem}

\begin{proof}
Let $n$ be the dimension of $\mathcal{A}$, and let $\Delta$ be a 2-local derivation of $\mathcal{A}$. It is enough to
show that $\Delta$ is linear. For every $x$, $y\in A$ there is an (inner) derivation $D_{x,y}$ of $\mathcal{A}$ such that
$\Delta(x)= D_{x,y}(x)$ and $\Delta(y)= D_{x,y}(y)$. By Proposition \ref{5.4}, the matrix of $D_{x,y}$ is skew-symmetric,
hence it is $\Lambda$-symmetric for $\Lambda=(\lambda_{i,j})_{i=1,\dots,n}$, with $\lambda_{i,j}=1$ if $1\leq i\leq j\leq n$,
and $\lambda_{i,j}=-1$ if $1\leq j<i\leq n$. Then, by Theorem \ref{0.2}, the proof is complete.
\end{proof}

As a part of the previous theorem we have the following theorem for
a finite dimensional exceptional Jordan algebra.

\begin{theorem} \label{5.51}
Every 2-local derivation of a finite dimensional formally real exceptional simple Jordan algebra
$\mathcal{A}$ is a derivation.
\end{theorem}

Let $s$ be a symmetry of $\mathcal{A}$, i.e., $s\in \mathcal{A}$ and $s^2 =1$, and let $U_s$ be the automorphism of $\mathcal{A}$ defined
by $U_s(x)=2s(sx)-x$ for $x\in \mathcal{A}$, i.e., $U_s =2L(s)^2-i_\mathcal{A}$.

\begin{proposition} \label{5.6}
Let $\mathcal{A}$ be a finite dimensional formally real Jordan algebra. Then for every symmetry
$s$ of $\mathcal{A}$ the automorphism $U_s$ of $\mathcal{A}$ is symmetric, and its matrix is symmetric.
\end{proposition}

\begin{proof}
It follows from item 3 of Lemma \ref{5.3} that $L(s)^2$ is symmetric. Then also $U_s =2L(s)^2-i_\mathcal{A}$ is
symmetric, and its matrix is symmetric.
\end{proof}

\begin{definition}
A map $\Delta$ on $\mathcal{A}$ is called a 2-local 1-automorphism, if
for every $x$, $y\in \mathcal{A}$ there exists a symmetry $s$ in $\mathcal{A}$ such that
$\Delta(x)=U_s(x)$, $\Delta(y)=U_s(y)$.
\end{definition}

\begin{theorem} \label{5.7}
Every 2-local 1-automorphism of a finite dimensional formally real Jordan algebra
$\mathcal{A}$ is an automorphism.
\end{theorem}

\begin{proof}
Let $n$ be the dimension of $\mathcal{A}$, and let $\Delta$ be a 2-local 1-automorphism of $\mathcal{A}$. It is enough to
show that $\Delta$ is linear. For every $x$, $y\in \mathcal{A}$, there is a symmetry $s\in \mathcal{A}$ such that $\Delta(x)=U_s(x)$
and $\Delta(y)=U_s(y)$. By Proposition \ref{5.6}, the matrix of $U_s$ is symmetric, hence it is $\Lambda$-symmetric for
$\Lambda=(\lambda_{i,j})_{i=1,\dots,n}$, with $\lambda_{i,j}=1$, if $1\leq i, j\leq n$. Then we conclude by Theorem \ref{0.2}.
\end{proof}

Thus as a part of the previous theorem we have the following theorem for
a finite dimensional exceptional Jordan algebra.

\begin{theorem} \label{5.71}
Every 2-local 1-automorphism of a finite dimensional formally real exceptional simple Jordan algebra
$\mathcal{A}$ is an automorphism.
\end{theorem}

\begin{remark} \label{5.8}
The same method applies for 2-local derivations and 2-local 1-automorphism of a
finite dimensional semisimple Jordan algebra $\mathcal{A}$ over a field of characteristic zero.
\end{remark}

{\bf Acknowledgements.}
The authors are indebted to the referees for the valuable comments and suggestions.
We thank one of the anonymous referees for many valuable suggestions and an important hint which allowed give
shorter proofs of the theorems in section \ref{formally real}.

\end{document}